\documentclass[12pt]{amsart}
\usepackage{amssymb,amsmath,amsfonts,latexsym,mathrsfs}
\usepackage{bm}
\usepackage{multirow,tabularx}

\usepackage{array,graphics,color}
\numberwithin{equation}{section}
\newcolumntype{Y}{>{\centering\arraybackslash}X}

\newtheorem{propn}{Proposition}[section]
\newtheorem{thm}[propn]{Theorem}
\newtheorem{lem}[propn]{Lemma}

\newtheorem{proposition}[propn]{Proposition}

\newtheorem*{thm*}{Theorem}
\theoremstyle{definition}
\newtheorem{defn}[propn]{Definition}

\newtheorem{rem}{Remark}[section]

\newtheorem*{question}{Question}

 \newcommand{\D}{\mathbb{D}}

\newcommand{\clh}{\mathcal{H}}

\begin{document}
\today

\title[A note on the column-row property]{A note on the column-row property}

\author[Ray]{Samya Kumar Ray}
\address{Stat-Math Unit, Indian Statistical Institute Kolkata, Kolkata, 700108, India}
\email{samyaray7777@gmail.com}

\author[Sarkar]{Srijan Sarkar}
\address{Department of Mathematics, Indian Institute of Science, Bangalore, 560012, India}
\email{srijans@iisc.ac.in,
srijansarkar@gmail.com}

	\subjclass[2010]{47L25: Operator spaces (= matricially normed spaces)}

\keywords{Column-row property, Operator spaces}

\begin{abstract}
In this article, we study the following question asked by Michael Hartz in a recent paper \cite{Hartz}: \textit{which operator spaces satisfy the column-row property?} We provide a complete classification of the column-row property for non-commutative $L_{p}$-spaces over semifinite von Neumann algebras. We study other relevant properties of operator spaces that are related to the column-row property and discuss their existence and non-existence for various natural examples of operator spaces.
\end{abstract}
\maketitle
\section{Introduction}
In \cite{Trent}, while studying the Corona problem for Dirichlet spaces on the unit disc $\D$, the author discovered an important property involving the corresponding multiplier algebra, which in recent times, is known as the \textit{column-row} property. The column-row property has emerged as an important tool in extending classical results on Hardy spaces to complete Nevanlinna-Pick (cnp) spaces.  In a remarkable recent work \cite{Hartz}, Hartz showed that every normalized cnp space has the column-row property with constant $1$. The notion of column-row property has led to a plethora of important results for cnp spaces, to name a few: (a) factorization for weak-product spaces; (b) interpolating sequences; (c) Corona problem etc.  (see \cite{AHMR2,AHMR, CH, Hartz,Trent}). Motivated by the inherent operator space structure of multiplier algebras and the immense application of this property,  Hartz asked the following question in \cite{Hartz}: 
\begin{question}
\textit{Which operator spaces satisfy the column-row property?}
\end{question}
Our aim in this article is to initiate a study for this question in the general setting of operator spaces, by looking at several examples. For the theory of operator spaces and related important results we refer to the excellent monographs \cite{ER00,paulsenbook, P03}. Let us begin with the description of column-row property by considering a \emph{concrete} operator space $E \subseteq B(\clh)$. 
For a sequence $\mathbf e:=(e_1,\dots,e_n)$ in $E,$ we define the column operator 
$C_{\mathbf e}
:\mathcal H\to\mathcal H^n$ by
$C_{\mathbf e}(\zeta):=[e_1(\zeta),\dots,e_n(\zeta)],\ \zeta\in\mathcal H.$
Similarly, let us define the corresponding row operator $R_{\mathbf e}:\mathcal H^n\to\mathcal H $ by 
$
R_{\mathbf{e}}([
\zeta_1,\dots,\zeta_n]):=\sum_{j=1}^ne_j(\zeta_j).
$
\begin{defn}\label{CRPde}
	A concrete operator space $E\subseteq B(\mathcal H)$ is said to have the \textit{column-row property} (in short CRP) if there exists a constant $C>0$ such that for any finite sequence $\mathbf e$ in $E$ with $\|C_{\mathbf e}\|_{\mathcal{H}\to\mathcal{H}^n}\leq 1$ we have \[\|R_{\mathbf e}\|_{\mathcal{H}^n\to\mathcal{H}}\leq C.\]
\end{defn}
In this article we obtain the following characterization for CRP for non-commutative $L_p$-spaces.
\begin{thm}\label{vonth}
Let $\mathcal M$ be a semifinite von Neumann algebra. Let $1\leq p\neq 2\leq\infty.$ Then $L_p(\mathcal M)$ has CRP if and only if $\mathcal M$ is subhomogeneous. 
\end{thm}
Moreover, we show that if $L_p(\mathcal M)$ has CRP with constant $1$ for some $p\in[1,\infty)\setminus\{2\}$, then $\mathcal M$ must be an abelian von Neumann algebra. In the case of $p=2$, it is well known that $L_2(\mathcal M)$ is an operator Hilbert space \cite[Page 139]{P03} and any operator Hilbert space is completely isometric to its opposite \cite[Exercise 7.1, Page 146]{P03} (see also [Footnote 8, Page 246]\cite{FI99}). Moreover, using the description of the matricial norms, it is a straightforward observation that any operator Hilbert space satisfies a stronger condition that is, the column-matrix property (CMP) with constant 1 (see Lemma \ref{s2norm}). We provide a short explanation of this fact along with some related results just after Remark \ref{ohmrp}.

Note that Hartz proved a stronger result by showing that cnp spaces satisfy the column-matrix property \cite[Corollary 3.6]{Hartz}. Motivated by these results, we study the column-matrix property (in short CMP) for operators spaces via some naturally occuring examples. We also introduce the completely bounded version of CRP and CMP. It was shown in \cite[Section 5.1]{Hartz} that normalized cnp spaces do not have completely bounded CRP with constant $1.$ However, it turns out that no non-trivial operator space can have completely bounded CRP. We also discuss CRP for many naturally occurring operator spaces including $C^*$-algebras.
 
Here, we would like to point out that the operator spaces we have mainly considered are self-adjoint. Thus, this situation is different from the study of the column-row property of multiplier algebras of rkHs, which are very much asymmetrical in nature. We thank the anonymous reviewer for this remark.  The main advantage that self-adjointness brings into the problem is the concept of subhomogeneity.  It allows us to form a line of thought where if subhomogeneity fails then the operator space contains certain matrix algebras, where we can study the column-row property by means of computation of norms. However, these computations are not straightforward and are obtained via non-trivial estimate of the norm of certain matrix-valued rows and columns. The main difficulty for establishing Theorem \ref{vonth} lies in the fact that the operator space structure of non-commutative $L_p$-spaces are difficult to work with. To overcome this difficulty, we first compute the norm of a column with certain entries coming from $S_2.$ Here we use the operator Hilbert space structure of $S_2$ and the description of the norm given by Pisier in \cite{P96}. Following this, we get an upper bound for the norm of the column by using the method of complex interpolation. Then by exhibiting lower bounds of the rows, we obtain the desired estimates for $p>2.$ The approach for the case of $p<2$ differs in the following manner: here we estimate norms of columns with entries from $S_1$. This is done using certain duality relationship between non-commutative vector valued $L_p$-spaces which was developed by Pisier in \cite{P98}.

Let us now briefly discuss the manner in which the rest of the article has been organised.  Section \ref{crp_op} contains the characterization of CRP for non-commutative $L_p$-spaces. In section \ref{natu}, we have studied CRP and CMP and their completely bounded versions for some natural examples of operator spaces. 
\section{Column-Row property for non-commutative $L_p$-spaces}\label{crp_op}
We begin with some preliminaries of operator space theory. 
For various well-known concepts related to the operator space theory we refer to \cite{ER00} and \cite{P03}. 

Let $E$ be an operator space equipped with a matricial norm structure $(M_n(E),\|.\|_{M_n(E)})_{n\geq 1}$. Given $x=[x_{ij}]_{i,j=1}^n\in M_n(E)$ we denote the transpose map by
$
t(x):=[x_{ji}]_{i,j=1}^n.$ The transpose maps taking columns to rows or rows to columns are also denoted by $t.$ Sometimes we also denote it by $t_n$ to specify the dimension of the spaces under consideration. The \textit{opposite} operator space, $E^{op}$ is defined to be the same space as $E$, but with the following matricial norm: $\|[x_{ij}]_{i,j=1}^n\|_{M_n(E^{op})}:=\|[x_{ji}]_{i,j=1}^n\|_{M_n(E)}$, for all $[x_{ij}]_{i,j=1}^n\in M_n(E^{op})$. For an operator space $E$ we denote the conjugate  operator space by $\overline{E}$. Any element $x\in E$ corresponds to an element $\overline{x}\in\overline{E}.$ We refer \cite[Section 2.9]{P03} for the notion of conjugate operator space. Recall that a linear map $u:E\to F$ between operator spaces is called \textit{completely bounded} if $\sup\limits_{n\geq 1}\|id_{M_n}\otimes u\|_{M_n(E)\to M_n(F)}<\infty,$ where $id_E:E\to E$ is the identity map for any vector space $E$. In this case, one denotes $\|u\|_{cb}:=\sup\limits_{n\geq 1}\|id_{M_n}\otimes u\|_{M_n(E)\to M_n(F)}.$ The map $u$ is called a complete isometry if $id_{M_n}\otimes u$ is an isometry for all $n\geq 1.$ The Hilbert space tensor product of two Hilbert spaces $\mathcal H$ and $\mathcal K$ is denoted by $\mathcal H\otimes_2\mathcal{K}.$ Let $A$ and $B$ be $C^*$-algebras with faithful representations $\pi_{\mathcal{H}}$ and $\pi_{\mathcal{K}}$ into $B(\mathcal H)$ and $B(\mathcal K)$ respectively. Note that by natural identification $\pi_{\mathcal H}(A)\otimes \pi_{\mathcal K}(B)$ and hence $A\otimes B$ can be identified as a subalgebra of $B(\mathcal{H}\otimes_2\mathcal K)$. The minimal tensor product of $A$ and $B$ is defined to be the completion of the algebraic tensor product $A\otimes B$ with norm borrowed from $B(\mathcal{H}\otimes_2\mathcal K)$ and is denoted by $A\otimes_{\text{min}}B$. It is well-known that the $C^*$-algebra $A\otimes_{\text{min}}B$ thus obtained is independent of the representations $\mathcal{\pi}_{\mathcal H}$ and $\mathcal{\pi}_{\mathcal K}.$ We denote by $I_n$ to be the identity matrix in $M_n.$ We need the following lemma.
\begin{lem}\label{tensor}
Let $E$ be an operator space. Then $E$ has CRP if and only if there exists a constant $C>0$ such that for all $n\geq 1$, $\|t_n\otimes id_E\|_{M_{n,1}\check{\otimes}E\to M_{1,n}\check{\otimes}E}\leq C,$ where $\check{\otimes}$ denotes the operator space injective tensor product.
\end{lem}
\begin{proof}
From \cite[Corollary 8.1.3]{ER00} we note that
$M_{m,n}(E) = M_{m,n}\check{{\otimes}}E $, where $\check{\otimes}$ denotes the operator space injective tensor product of operator spaces.
\end{proof}


If $A$ is a $C^*$-algebra then let us fix a faithful $*$-representation of $A$ on ${B}(\mathcal{H})$. The canonical operator space structure on $A$ is obtained by borrowing the matricial structure from ${B}(\mathcal{H})$, and we will be using this canonical structure in the sequel. A $C^*$-algebra $A$ is called $n$-subhomogeneous if all the irreducible representations of $A$ have dimensions at most $n.$ A $C^*$-algebra is called subhomogeneous if it is $n$-subhomogeneous for some $n\in\mathbb N.$ Given a von Neumann algebra $(\mathcal M,\tau)$ with normal faithful semifinite trace $\tau $, let $L_p(\mathcal M,\tau)$ be the corresponding non-commutative $L_p$-space for $0<p<\infty$. One denotes $L_\infty(\mathcal M)=\mathcal M.$ When $(\mathcal M,\tau)=(B(\mathcal H),Tr)$ with $Tr$ being the natural trace on $B(\mathcal H),$ the corresponding non-commutative $L_p$-spaces are called the Schatten-$p$ classes and are denoted by $S_p(\mathcal H)$ for $1\leq p<\infty.$ The space $S_\infty(\mathcal H)$ is the space of all compact operators on $\mathcal H.$ These spaces are denoted by $S_p$ and $S_p^n$ when $\mathcal H$ is $\ell_2$ or $\ell_2^n$ respectively for $1\leq p\leq \infty.$ We often identify ${B}(\ell_2^n)$ with $M_n$. Note that if $\mathcal M$ is a von Neumann algebra with a normal faithful semifinite trace $\tau$, $M_n(\mathcal M)$ is again a von Neumann algebra equipped with the canonical tensor trace $Tr\otimes\tau$. Indeed we have a canonical identification of $M_n(\mathcal M)$ with the von Neumann algebra $M_n\overline{\otimes}\mathcal M$, where $M_n\overline{\otimes}\mathcal M$ is the von Neumann algebra tensor product of $M_n$ and $\mathcal M.$ Then $Tr\otimes \tau(\sum_{i=1}^Na_i\otimes x_i):=\sum_{i=1}^NTr(a_i)\tau(x_i)$ extends to a normal faithful semifinite trace on $M_n\overline{\otimes}\mathcal M$, where $a_i\in M_n$ and $x_i\in\mathcal M$ are positive elements for all $1\leq i\leq N$ and $N\geq 1.$ Equivalently for any positive $x\in M_n(\mathcal M)$, we have $Tr\otimes\tau(x)=\sum_{i=1}^n\tau(x_{ii}).$ Equipped with the canonical operator space structures $(\mathcal M_{*}, \mathcal M)$ becomes an interpolation couple (see \cite[Section 2.7 and Chapter 7]{P03} and \cite{FI99}), where $\mathcal{M}_{*}$ is the predual of $\mathcal M.$ We identify the predual ${\mathcal M}_{*}$ with $L_1(\mathcal M)$ via the map $\phi:L_1(\mathcal M)\to {\mathcal M}_{*}$ as $\phi(y)(x):=\tau(xy)$ for all $y\in L_1(\mathcal M)$ and $x\in\mathcal M.$ Moreover, $L_1(\mathcal M)$ has a natural operator space structure induced by $\mathcal M_{*}$ (see \cite[Page 139]{P03}).  We have the following description of the operator space structure of $L_p(\mathcal M,\tau)$  from the convention established in \cite{P03}. For $x\in M_n(L_p(\mathcal M,\tau)),$ we have for all $1\leq p\leq \infty$
\begin{eqnarray}\label{ncformula}
\|x\|_{M_n(L_p(\mathcal M,\tau))}=\sup\{\|axb\|_{L_p(M_n(\mathcal M ))}:\|a\|_{S_{2p}^n}\leq 1 \ \|b\|_{S_{2p}^n}\leq 1\},
\end{eqnarray}
where $axb$ is the usual product of matrices. It follows from \cite[Lemma 1.7]{P98} that for a map $u:L_{p}(\mathcal M)\to L_p(\mathcal N)$ we have $\|u\|_{cb, L_p(\mathcal M)\to L_p(\mathcal N)}=\|id_{S_p^n}\otimes u\|_{L_p(M_n\overline{\otimes}\mathcal M)\to L_p(M_n\overline{\otimes}\mathcal N)}.$ Moreover, $u$ is a complete isometry iff $id_{S_p^n}\otimes u$ is an isometry for all $n\geq 1.$
We refer \cite{Tak79} for the general theory of von Neumann algebras along with the notions of trace and tensor products of von Neumann algebras. We refer \cite{PiX03, Terp81} for non-commutative $L_p$-spaces.
\begin{rem}Note that the norm on $M_n(L_p(\mathcal M))$ is very \textit{different} from $L_p(M_n\overline{\otimes}\mathcal M)$ as can be seen from our computations in the proof of Theorem \ref{noncrp}.
\end{rem} 

We will now introduce a class of operator spaces which is important for this article. Following \cite[Theorem 1.1]{P96} (see also \cite[Theorem 7.1]{P03}), we can give the following definition of an operator Hilbert space denoted by $OH(I)$.
\begin{defn}
For any index set $I$, there always exists an unique operator space $OH(I)$ (up to complete isometry) such that the following properties are satisfied:
\begin{itemize}
\item[(i)] $OH(I)$ is isometric to $\ell_2(I)$ as a Banach space;
\item [(ii)] the canonical identification between $\ell_2(I)$ and $\overline{\ell_2(I)^*}$ induces a complete isometry from $OH(I)$ to $\overline{OH(I)^*}$.
\end{itemize}
This unique operator space $OH(I)$ is known as the operator Hilbert space.
Moreover, if $\mathcal K$ is a Hilbert space and $(T_i)_{i\in I}$ is any orthonormal basis of $OH(I)$ then for any finitely supported family $(x_i)_{i\in I}$ in $B(\mathcal K)$ we have that \[\Big\|\sum_{i\in I}x_i\otimes T_i\Big\|_{\text{min}}=\Big\|\sum_{i\in I}x_i\otimes \overline{x_i}\Big\|_{\text{min}}^{\frac{1}{2}},\] where $\|.\|_{\text{min}}$ denotes the norm of minimal tensor product between $C^*$-algebras.
\end{defn}
Thus, for any $[x_{ij}]_{i,j=1}^n\in M_n(OH(I))$, we have 
\begin{equation}\label{ohsds}
\Big\|[x_{ij}]_{i,j=1}^n\Big\|_{M_n(OH(I))}=\Big\|\Big[ [\langle x_{ij}, x_{kl}\rangle]_{k,l=1}^n \Big]_{i,j=1}^n\Big\|^{\frac{1}{2}}_{M_{n^2}},
\end{equation}
(see for instance, \cite[Exercise 7.5]{P03} and \cite[Proposition 1]{FI99} for a proof). 
Since $L_2(\mathcal M)$ is completely isometrically isomorphic to an operator Hilbert space \cite[Page 139]{P03} (also see \cite[Page 125]{P03}), we can describe the matricial structure of $L_2(\mathcal M)$ by the same formula as \eqref{ohsds} for any $[x_{ij}]_{i,j=1}^n\in M_n(L_2(\mathcal M))$ with the inner product $\langle x,y\rangle:=\tau(xy^*)$ for all $x,y\in L_2(\mathcal M).$ It is well known that $\|id_{OH(I)}\|_{cb,OH(I)\to OH(I)^{op}}=1$ (see \cite{FI99}).  We state the following lemma without giving the straightforward proof.
\begin{lem}\label{s2norm}
Let $a_i\in OH(I),$ $1\leq i\leq n.$ Then $\|[a_1\dots a_n]^t\|_{M_{n,1}(OH(I))}=\Big(\sum_{i,j=1}^n|\langle a_i,a_j\rangle|^2\Big)^{\frac{1}{4}}.$
\end{lem}

In the sequel, we shall frequently use the following block matrices for proving our results.
\[
A_n: = \begin{bmatrix}
E_{11} & E_{12} & \dots & E_{1n}\\
0 & 0 & \dots & 0\\
\vdots & \vdots & \ddots & \vdots\\
0 & 0 & \dots & 0
\end{bmatrix}_{n \times n}; \quad B_n: = \begin{bmatrix}
	E_{11} & E_{21} & \dots & E_{n1}\\
	0 & 0 & \dots & 0\\
	\vdots & \vdots & \ddots & \vdots\\
	0 & 0 & \dots & 0
	\end{bmatrix}_{n \times n}.
\]
Now let us record a straightforward observation that will be useful in the sequel.
\begin{lem}\label{Sper}Let $1\leq p<\infty.$ Then
$\| A_n \|_{S_p^{n^2}}=\sqrt{n}\ \text{and}\ \|B_n\|_{S_p^{n^2}}=n^{\frac{1}{p}}.$
\end{lem}

\begin{lem}\label{S1er} Let $y=[y_{kl}]_{k,l=1}^n\in M_n(S_\infty^n)$ such that $\|y\|_{M_n(S_\infty^n)}\leq 1.$ Then for all $a\in S_2^n$ with $\|a\|_2\leq 1,$ we have 
\[
\sum_{l=1}^n\Big|\sum_{j,k=1}^na_{jk}y_{j1}^{kl}\Big|^2\leq 1,
\]
where $y_{kl}=[y_{ij}^{kl}]_{i,j=1}^n.$
\end{lem}
\begin{proof}
Note that $z=[z_{lk}]_{l,k=1}^n$ has norm $\leq 1,$ where $z_{lk}=y_{kl}^*.$ For all $1\leq k \leq n$, consider $v_k:=(v_{1k}, \ldots, v_{nk})\in \ell_2^n$ such that $\|v\|_{\ell_2^{n^2}}\leq 1$, where $v:=(v_1,\ldots,v_n).$ Note that we must have $\|zv\|_{\ell_2^{n^2}}^2\leq 1$ as well.  Hence we get
\begin{equation}\label{kichuek}
\sum_{l=1}^n\Big\|\sum_{k=1}^nz_{lk}v_k\Big\|^2\leq 1.
\end{equation}
Note that $z_{lk}v_k=(\sum_{j=1}^nz_{ij}^{lk}v_{jk})_{i=1}^n$, where $z_{lk} = [z_{ij}^{lk}]_{i,j=1}^n$. Therefore, we get $\sum_{k=1}^nz_{lk}v_k=(\sum_{k=1}^n\sum_{j=1}^nz_{ij}^{lk}v_{jk})_{i=1}^n$. Thus,  form equation \ref{kichuek}, we obtain that
\[
\sum_{l=1}^n\sum_{i=1}^n\Big|\sum_{k=1}^n\sum_{j=1}^nz_{ij}^{lk}v_{jk}\Big|^2\leq 1.
\]
Fixing $i=1$ in the above inequality we get 
$
\sum_{l=1}^n\Big|\sum_{k=1}^n\sum_{j=1}^nz_{1j}^{lk}v_{jk}\Big|^2\leq 1.
$
Rewriting we get
$\sum_{l=1}^n\Big|\sum_{k=1}^n\sum_{j=1}^ny_{j1}^{kl}\overline{v_{jk}}\Big|^2\leq 1.$ Putting $\overline{v_{jk}}=a_{jk}$ for $1\leq j,k\leq n$, we obtain the desired inequality.
\end{proof}
\begin{lem}\label{noncrp}
Let $1\leq p\neq 2\leq \infty.$ Then $\|t_n\otimes id_{S_p^n}\|_{M_{n,1}\check{\otimes}S_p^n\to M_{1,n}\check{\otimes}S_p^n}\geq n^{\frac{|p-2|}{2p}}.$
\end{lem}
\begin{proof}  
Let us fix $2<p<\infty$ and consider $A_n \in M_n(S_p^n).$
Then by choosing 
\begin{equation}\label{choice2}
a=E_{11}; \quad \ b=n^{-\frac{1}{2p}}I_n, 
\end{equation} in the formula \eqref{ncformula}, we have the estimate
$\|A_n\|_{M_n(S_p^n)}\geq \|a A_n b\|_{L_p(M_n(B(\ell_2^n)))}.$ By this and Lemma \ref{Sper} we obtain that $\|A_n\|_{M_n(S_p^n)}\geq n^{-\frac{1}{2p}}\|A_n\|_{S_p^{n^2}}=n^{\frac{1}{2}-\frac{1}{2p}}=n^{\frac{1}{2p^\prime}},$ where $\frac{1}{p}+\frac{1}{p^\prime}=1.$ 

Now we shall estimate $\|t(A_n)\|_{M_n(S_p^n)}$. We know that 
$\|t(A_n)\|_{M_n(B(\ell_2^n))}=1.$ We start with calculating $\|t(A_n)\|_{M_n(S_2^n)}$. Note that the operator space structure of $S_2^n$ agrees with the operator Hilbert space structure \cite{FI99}. Therefore, by \cite[Theorem 7.1]{P03} and Lemma \ref{s2norm}, we get
	\begin{equation}\label{comput}
	\|t(A_n)\|_{M_n(S_2^n)}=(\sum_{i=1}^n\sum_{j=1}^n|\langle E_{1i},E_{1j}\rangle|^2)^{\frac{1}{4}}=(\sum_{i=1}^n|\langle E_{1i},E_{1i}\rangle|^2)^{\frac{1}{4}}=n^{\frac{1}{4}}.
\end{equation}	
Now we shall use the method of complex interpolation (for reference, see \cite[Section 2.7, Page 52]{P03} and \cite[Chapter 4, Page 87]{BL}). Consider $\theta \in (0,1)$ be such that $\frac{1}{p} = \frac{\theta}{2}+0$. Thus, we get 
\[
\|t(A_n)\|_{M_n(S_p^n)}\leq (n^{\frac{1}{4}})^{\frac{2}{p}} = n^{\frac{1}{2 p}} . 
\]
Hence we obtain the estimate
\[\frac{\|A_n\|}{\|t(A_n)\|}\geq\frac{n^{\frac{1}{2p^\prime}}}{n^{\frac{1}{2p}}}=n^{\frac{p-2}{2p}}.\]
Now we will study the case for $p \in [1,2)$, for which let us consider $B_n \in M_n(S_p^n).$
Note that by choosing $a$ and $b$ as in \eqref{choice2} in the formula (\ref{ncformula}), we have the estimate
	$\|B_n\|_{M_n(S_p^n)}\geq \|aB_nb\|_{L_p(M_n(B(\ell_2^n)))}= n^{\frac{1}{2p}}.$ Now let us estimate $\|t(B_n)\|_{M_n(S_p^n)}.$ To do this we use the method of complex interpolation again. It is easy to see that $ \| t(B_n) \|_{M_n(S_1^n)}=\|A_n\|_{M_n(S_1^n)}$.
For the sake of computation, let us denote the latter block matrix by $\bm{z}:=(z_{ij})_{i,j=1}^n$ that is, $z_{1i}=E_{1i}$ fo $1\leq i\leq n$ and $z_{ij}=0$ for $i\geq 2.$ 

Now we will estimate $\|\bm{z}\|_{M_n(S_1^n)}.$ Following the method of \cite[Lemma 1.7]{P98},
we get that $
\|\bm{z}\|_{M_n[S_1^n]} = \sup \{ \|a \cdot [z_{ij}] \cdot b \|_{S_1^n[S_1^n]}: a,b \in S_{2}^n, \|a\|_{S_2^n}, \|b\|_{S_2^n} \leq 1\}.
$
Using the fact, $M_n[S_1^n]^* = S_1^n[S_{\infty}^n]$ and \cite[Theorem 1.5]{P98}, we observe that 
\[
\|\bm{z}\|_{M_n(S_1^n)} = \sup \{ \big | \langle [z_{ij}], a \cdot [y_{ij}] \cdot b \rangle \big | : a,b \in S_{2}^n, \|a\|_{S_2^n}, \|b\|_{S_2^n}, \|[y_{ij}]\|_{M_n[S_{\infty}^n]} \leq 1\}.
\]
Now, $ a\cdot [y_{ij}] \cdot b = a \cdot [\sum_{l=1}^n y_{il}b_{lj}]_{i,j=1}^n = [\sum_{k=1}^n a_{ik} (\sum_{l=1}^n y_{kl} b_{lj})]_{i,j=1}^n = [\sum_{k,l=1}^n a_{ik} y_{kl} b_{lj}]_{i,j=1}^n
$.  
By definition,
\begin{align*}
\langle [z_{ij}], a \cdot [y_{ij}] \cdot b \rangle = \sum_{i,j=1}^n \langle z_{ij}, \sum_{l,k=1}^n a_{jk} y_{kl} b_{li} \rangle =  \sum_{i,j,k,l=1}^n a_{jk}  \langle z_{ij}, y_{kl} \rangle b_{li}.
\end{align*}
Hence,
\[
\|\bm{z}\|_{M_n(S_1^n)} = \sup \{ \big |  \sum_{i,j,k,l=1}^n a_{jk} \langle z_{ij}, y_{kl} \rangle b_{li}\big | : a,b \in S_{2}^n, \|a\|_{S_2^n}, \|b\|_{S_2^n}, \|[y_{ij}]\|_{M_n[S_{\infty}^n]} \leq 1\}. 
\]
Therefore, from the definition of $\bm{z}$, we get
\[
 \|\bm{z}\|_{M_n(S_1^n)} = \sup \{\big | \sum_{j,k,l=1}^n a_{jk}  \langle E_{1j}, y_{kl} \rangle b_{l1} \big | : a,b \in S_{2}^n, \|a\|_{S_2^n}, \|b\|_{S_2^n}, \|[y_{ij}]\|_{M_n[S_{\infty}^n]} \leq 1\}.
\]
Now, $ \langle E_{1j}, y_{kl} \rangle = y_{j1}^{kl}$, where $y_{kl}=[y_{ij}^{kl}]_{i,j=1}^n$ and therefore, 
\begin{align*}
\|\bm{z}\|_{M_n(S_1^n)} &= \sup \{  \big | \sum_{j,k,l=1}^n a_{jk}  y_{j1}^{kl} b_{l1} \big | : a,b \in S_{2}^n, \|a\|_{S_2^n}, \|b\|_{S_2^n}, \|[y_{ij}]\|_{M_n[S_{\infty}^n]} \leq 1\}\\
&= \sup \{   (\sum_{l=1}^n | \sum_{j,k=1}^n a_{jk}y_{j1}^{kl} |^2)^{\frac{1}{2}} : a \in S_{2}^n, \|a\|_{S_2^n}, \|[y_{ij}]\|_{M_n[S_{\infty}^n]} \leq 1\},
\end{align*}
where the last equality is obtained by taking supremum over $\|b\|_{S_2^n} \leq 1$.  Now using Lemma \ref{S1er}, we obtain $\|\bm{z}\|_{M_n[S_1^n]}\leq 1$ and thus, $\|t(B_n)\|_{M_n[S_1^n]} \leq 1$.  If we do computations in a manner similar to condition (\ref{comput}), we get $\|t(B_n)\|_{M_n[S_2^n]} = n^{\frac{1}{4}}$. Therefore, by the method of complex interpolation we get
$
\|t(B_n)\|_{M_n[S_p^n]}\leq n^{\frac{1}{2p^{\prime}}} \quad (1\leq p<2).
$
Thus, by combining all the facts obtained above for any $p \in [1,2)$, we get
\begin{eqnarray}
\frac{\|B_n\|}{\|t(B_n)\|}\geq\frac{n^{\frac{1}{2p}}}{n^{\frac{1}{2p^{\prime}}}}=n^{\frac{2-p}{2p}}.
\end{eqnarray}
 This completes the proof.
\end{proof}
\begin{rem}\label{somesti}
Let $1\leq p\neq 2\leq\infty.$ In view of \cite[Lemma 5.3]{LZ221} and Lemma \ref{noncrp} we have the estimate  $n^{\frac{|p-2|}{2p}}\leq\|t_n \otimes id_{S_p^n}\|_{M_{n,1}\check{\otimes}S_p^n\to M_{1,n}\check{\otimes}S_p^n}\leq n^{\frac{|p-2|}{p}}.$  We believe that the lower bound is actually sharp.
\end{rem}
\begin{rem}Let $E$ be an operator space. Then we have the following general estimate. 
\begin{align*}
\|t_n\otimes id_{E}\|_{M_{n,1}\check{\otimes}E\to M_{1,n}\check{\otimes}E} &\leq \|t_n\otimes id_{E}\|_{cb,M_{n,1}\check{\otimes}E\to M_{1,n}\check{\otimes}E} \\
&\leq\|t_n\|_{cb,M_{n,1}\to M_{1,n}}\|id_{E}\|_{cb,E\to E}\leq\sqrt{n}.
\end{align*}
In the above we have used the estimate $\|t_n\|_{cb,M_{n,1}\to M_{1,n}}=\sqrt{n}$ (see \cite[Page 22]{P03}).
\end{rem}

We are now ready to prove Theorem \ref{vonth} in the following manner.
\begin{proof}
See Theorem \ref{staral} for $p=\infty.$ Let $p\neq\infty.$ Let $\mathcal M$  be subhomogeneous. Then it follows from \cite[Lemma 4.4]{LZ22} that $L_p(\mathcal M)$ has CRP. For the converse, let us assume that $\mathcal M$ is not subhomogeneous. Then by \cite[Lemma 2.1]{LZ22} it follows that for all $n\geq 1$ there is a complete isometry from $S_p^n$ into $L_p(\mathcal M).$ Hence by Lemma \ref{noncrp} and Lemma \ref{tensor}, $L_p(\mathcal M)$ cannot have CRP for $p\neq 2.$ This completes the proof of the theorem.
\end{proof}

\begin{proposition}Let $\mathcal M$ be a semifinite von Neumann algebra. Let $1\leq p\neq 2\leq \infty$. If $L_p(\mathcal M)$ has CRP with constant $1$ then $\mathcal M$ must be an abelian von Neumann algebra.
\end{proposition}
\begin{proof}
See Theorem \ref{staral} for $p=\infty.$ Note that by Theorem \ref{vonth}, it follows that $\mathcal M$ must be subhomogeneous. By \cite[Lemma 2.1]{LZ22} if $\mathcal M$ is not subhomogeneous of degree $1,$ it contains $S_p^2$ completely isometrically. Therefore, by Lemma \ref{noncrp} we have the estimate
 $\|t_n\otimes id_{S_p^2}\|_{M_{2,1}\check{\otimes}S_p^2\to M_{1,2}\check{\otimes}S_p^n}\geq 2^{\frac{|p-2|}{2p}}>1.$ This shows that $\mathcal M$ must be subhomogeneous of degree $1.$ Hence by \cite[Theorem 7.1.1]{Run20}, we get that $\mathcal M$ is an abelian von Neumann algebra.
\end{proof}
\section{Column-row property for other operator spaces}\label{natu}
We shall now study the completely bounded versions of the notions CRP and CMP. To state our result we introduce the following definitions.
\begin{defn}
Let $E$ be an operator space. We say $E$ has the \textit{column-matrix property} (in short CMP) if there exists a constant $C>0$ such that for all $n\in\mathbb N$ and $[x_{ij}]_{i,j=1}^n\in M_n(E),$ we have 
\[
\|[x_{ij}]_{i,j=1}^n\|_{M_n(E)}\leq C\| \begin{bmatrix}
\vspace{-2mm}
x_{11} \\
\vspace{-2mm}
\vdots\\
\vspace{-3mm}
x_{n1}\\
\vspace{-2mm}
\vdots\\
x_{nn}
\end{bmatrix}\|_{M_n(E)}.
\]
\end{defn}
\begin{defn}
Let $E$ be an operator space. $E$ is said to satisfy the completely bounded CRP if there exists a constant $C>0$ such that for all $n\geq 1$, we have \[\|[x_1\dots x_n]^t\mapsto [x_1\dots x_n]\|_{cb, M_{n,1}(E)\to M_{1,n}(E)}<C.\]
\end{defn}
We now show that there is no non-trivial operator space with the completely bounded CRP.
\begin{proposition}\label{lastth}
Let $E$ be an operator space which is not equal to the zero vector space. Then $E$ does not have the completely bounded CRP.
\end{proposition} 

\begin{proof}
Let $e\in E$ be such that $\|e\|=1.$ Then it follows from the property of the operator space injective tensor product \cite[Page 142]{ER00} that the map $i_n:M_{n,1}\to M_{n,1}\check{\otimes}E$ given by $x\mapsto x\otimes e$ is a complete isometry. Also $t_n(x\otimes e)=t_n(x)\otimes e$ for all $x\in M_{n,1}.$   We have $\|t_n\|_{cb, M_{n,1}\check{\otimes}E\to M_{1,n}\check{\otimes}E}\geq \sup_{N\geq 1,[x_{ij}]_{i,j=1}^N\neq 0}\frac{\|[t_n(x_{ij})]_{i,j=1}^N\|_{M_N(M_{1,n})}}{\|[x_{ij}]_{i,j=1}^N\|_{M_N(M_{n,1})}}=\sqrt{n}.$ The last estimate in the above inequality follows from \cite[Page 22]{P03}. This shows that $E$ does not have completely bounded CRP. This completes the proof of the theorem.
\end{proof}
We present a characterization of CRP for $C^*$-algebras. First we need the following lemma.
\begin{lem}\label{double dual}Let $E$ be an operator space. Then $E$ has CRP iff $E^{**}$ has CRP.
\end{lem}
\begin{proof}
Note that if $E^{**}$ has CRP, then $E$ has CRP follows from the fact that the natural inclusion of $E$ inside $E^{**}$ is a complete isometry \cite[Proposition 3.2.1]{ER00}. To prove the converse, note by the duality between the projective and injective tensor product of operator spaces (see \cite[Theorem 4.1]{P03}) and the fact that $M_{n,1}^*=M_{1,n}$ completely isometrically (see \cite[Exercise 2.3.5]{P03}), we have $
\|t_n\otimes id_{E^{**}}\|_{M_{n,1}\check{\otimes}E^{**}\to M_{1,n}\check{\otimes}E^{**}}=\|t_n\otimes id_{E}\|_{M_{n,1}\check{\otimes}E\to M_{1,n}\check{\otimes}E}.$ The proof follows by using Lemma \ref{tensor}.
\end{proof}
\begin{thm}\label{staral} Let $A$ be a $C^*$-algebra. Then $A$ has CRP iff $A$ is subhomogeneous. Moreover, $A$ has CRP with constant $1$ iff $A$ is abelian.
\end{thm}
\begin{proof}
One direction follows from \cite{Roy05}. Note that if $A$ has CRP, then by Lemma \ref{double dual} $A^{**}$ is a von Neumann algebra which has CRP. However, it is folklore that if $A^{**}$ is not subhomogeneous then it contains $M_n$ for all $n\geq 1.$ We get a contradiction by looking at $A_n$ and $B_n$ defined in Section 2.
\end{proof}
\begin{rem}One can also have the following alternative approach. Recall that the Haagerup tensor product norm on the algebraic tensor product  is defined by 
\[
\|z\|_h:=\inf\Big\{\Big\|\sum_{i=1}^n x_ix_i^{*}\Big\|^{\frac{1}{2}}\Big\|\sum_{i=1}^n y_i^*y_i\Big\|^{\frac{1}{2}}:z=\sum_{i=1}^n x_i\otimes y_i\Big\} \quad (z \in A\otimes A).
\] 
Let us consider $\text{inv}:A\otimes A\to A\otimes A$, defined on the elementary tensors by $\text{inv}(x\otimes y ):=x^*\otimes y^*.$ If $A$ has CRP it follows that for some $C>0$ and for all $z\in A\otimes A,$ $ \|\text{inv}(z)\|\leq C\|z\|.
$
The result follows from \cite[Part (iii) Theorem 2.2.]{ajay}. If $C=1,$ then $A$ is abelian follows from \cite[Theorem 2.3]{ajay}. 
\end{rem}
\begin{rem}\label{ohmrp} Operator spaces for which $\|id_E\|_{cb, E\to E^{op}}=1$, have been studied for $C^*$-algebras and operator algebras in \cite{Oka70} and \cite{Ble90}.
\end{rem}
The following list summarizes the results which establishes CRP and other related properties for some well known examples of operator spaces. Checking the properties are easy in the most of the cases and some of it maybe present in the literature.
{\begin{enumerate}
\item[$\bullet$] $MIN(E), MAX(E)$ and $L_p(\Omega)$ for $1\leq p\leq \infty,$ have CRP with constant $1$, where $MIN(E)$ and $MAX(E)$ denote the so called \textit{minimal and maximal quantization of a Banach space $E$}. For $MIN(E)$ and $MAX(E)$ this follows from their definitions. By duality and interpolation one sees that the result for $L_p(\Omega)$ follows from complex interpolation.  Moreover, $MIN(E)$ has CMP with constant $1.$ This is because $\text{MIN}(E)$ embeds completely isometrically into $C(K)$ for some compact Hausdroff topological space $K$ \cite[Proposition 3.3.1]{ER00} and for all $[f_{ij}]_{i,j=1}^n\in M_n(C(K)),$ we have $\sup_{s\in K}\|[f_{ij}(s)]\|_{op}\leq\sup_{s\in K}(\sum_{i,j=1}^n|f_{ij}(s)|^2)^{1/2}.$ 
\item[$\bullet$] The operator Hilbert space over the indexing set $I$, denoted by $OH(I)$, have CMP with constant $1$, which further implies that $S_2$ (the class of Hilbert-Schmidt operators on $\ell_2$) and more generally $L_2(\mathcal M)$ have this property as well with constant $1$, where $\mathcal M$ is a semifinite von Neumann algebra. Note that by the description of the norm of $OH(I)$ we have the formula \[\|[x_{ij}]_{i,j=1}^n\|_{M_n(OH(I))}=\|\big [[\langle x_{ij}, x_{kl}\rangle]_{k,l=1}^n \big ]_{i,j=1}^n\|^{\frac{1}{2}}_{M_{n^2}}.\] From Lemma \ref{s2norm}, we get $\|[a_1 \ldots a_n]^t\|_{M_{n,1}(OH(I))}=(\sum_{i,j=1}^n|\langle a_i, a_j\rangle|^2)^{\frac{1}{4}}$, and thus we get the trivial inequality 
\[
\|[x_{ij}]_{i,j=1}^n\|_{M_n(OH(I))}\leq (\sum_{i,j=1}^n\sum_{k,l=1}^n|\langle x_{ij},x_{kl}\rangle|^2)^{\frac{1}{2}}.
\]
\item[$\bullet$]The column Hilbert space $C$ has the column-row property with constant $1$. The row Hilbert space does not have the column-row property.
\end{enumerate}

\textit{Concluding remarks}:
In the literature, there are several important studies on properties which have a certain resemblance to CRP. An operator space is called \textit{symmetric} if it is completely isometric to its opposite \cite{Ble90} via the identity map. Okayusu \cite{Oka70} showed that a $C^*$-algebra is symmetric if and only if it is commutative. Thus Theorem \ref{staral}, can be considered as a generalization of the afore-mentioned result.  Using Okayusu's characterization, Tomiyama \cite{Tom83} studied the positivity of the transpose map on $C^*$-algebras. Furthermore, Blecher \cite{Ble90} extended Okayusu's result and showed that any unital operator algebra which is symmetric must be commutative. We refer the reader to \cite{Roy05} for some related results. Our characterizations on $L_p$-spaces with CRP can be observed as a contribution towards results belonging to the theme of the above direction. Moreover, estimates on the completely bounded norm of the transpose maps on non-commutative $L_p$-spaces were crucial in the recent work of Le Merdy and Zadeh for studying separating maps on non-commutative $L_p$-spaces (see \cite{LZ22} and \cite{LZ221}). 

\section*{Acknowledgement}
The authors are thankful to Michael Hartz for some useful communications throughout the work. They would also like to thank the anonymous referee for reading the manuscript intricately and giving important comments and  revisions which have improved the presentation of this article. The first and second author acknowledges the DST-INSPIRE Faculty Fellowship DST/INSPIRE/04/2020/001132 and DST/INSPIRE/04/2019/000769, respectively.

\end{document}